\documentclass[12pt, reqno]{amsart}
\usepackage{}
\usepackage{cases}
\usepackage{txfonts}
\usepackage{amsmath, amsthm, amscd, amsfonts, amssymb, graphicx, color}
\usepackage[bookmarksnumbered, colorlinks, plainpages]{hyperref}

\newtheorem{theorem}{Theorem}[section]
\newtheorem{lemma}[theorem]{Lemma}
\newtheorem{proposition}[theorem]{Proposition}
\newtheorem{corollary}[theorem]{Corollary}
\theoremstyle{definition}
\newtheorem{definition}[theorem]{Definition}
\newtheorem{example}[theorem]{Example}

\newtheorem{fact}[theorem]{Fact}

\newtheorem{problem}[theorem]{Problem}
\theoremstyle{remark}
\newtheorem{remark}[theorem]{Remark}

\newcommand{\N}{\mathbb{N}}
\newcommand{\eps}{\varepsilon}

\numberwithin{equation}{section} 
\begin{document}
\setcounter{page}{1}

\title[Stability of Banach spaces via nonlinear $\varepsilon$-isometries]{{\bf Stability of Banach spaces via nonlinear $\varepsilon$-isometries }}

\author[D. Dai, Y. Dong]{Duanxu Dai$^1$ and Yunbai Dong$^{*,2}$}

\address{$^{1}$ School of Mathematical Sciences, Xiamen University, Xiamen 361005, P.R. China.} \email{\textcolor[rgb]{0.00,0.00,0.84}{dduanxu@163.com}}

\address{$^{1}$ Present address: Department of Mathematics, Texas A\&M University, College Station, TX 77843-3368, USA.}

\address{$^{2}$ School of Mathematics and Computer, Wuhan Textile University, Wuhan 430073, P.R. China.}

\email{\textcolor[rgb]{0.00,0.00,0.84}{baiyunmu301@126.com}}

\thanks{$^*$ Corresponding author.}

%\dedicatory{This paper is dedicated to Professor ABCD}



\begin{abstract}
In this paper, we prove that the existence of an $\varepsilon$-isometry from a separable Banach space $X$ into $Y$ (the James space or a reflexive space) implies the existence of a linear isometry from $X$ into $Y$. Then we present a set valued mapping version lemma on non-surjective $\varepsilon$-isometries of Banach spaces. Using the above results, we also discuss the rotundity and smoothness of Banach spaces under the perturbation by $\varepsilon$-isometries.

\end{abstract} \maketitle

\section{Introduction}

Throughout the paper $X$ and $Y$ denote real Banach spaces. An $isometry$ from $X$ to $Y$ is a mapping $f: X\rightarrow Y$ such that
$\|f(x)-f(y)\|=\|x-y\|$
for all $x,y\in X$.

For surjective isometries, Mazur and Ulam \cite{Maz} have given a
 result. They showed that if $f$ is a surjective isometry
between two real Banach spaces, then $f$ is affine. While a
non-surjective isometry is not necessarily affine, for example, defining $f:R\rightarrow\ell_\infty^2$ by $f(t)=(t,sint)$, $t\in R$.

In 1967, Figiel \cite {Fig} proved the following remarkable result
on non-surjective isometries, which is an appropriate substitute of
the Mazur-Ulam theorem. He showed that for any isometry $f:X\rightarrow Y$ with $f(0)=0$ there is a linear operator $P:\overline{\rm span} f(X)\rightarrow X$ of norm one such that $P\circ f$ is the identity on $X$.

We next recall the following concept which is related to isometries
of Banach spaces.

\begin{definition}
Given $\varepsilon>0$, a mapping $f:X\rightarrow Y$ is called an
$\varepsilon-isometry$ if
$$\big|\|f(x)-f(y)\|-\|x-y\|\big|\leq\varepsilon
\text{~~~~for~~all}~~x,y\in X.$$
\end{definition}

These mappings were first studied by Hyers and Ulam \cite {Hye}, and
they asked if for every surjective
$\varepsilon$-isometry $f:$ $X\rightarrow Y$ with $f(0)=0$ there
exists a surjective linear isometry $U:$ $X\rightarrow Y$ and
$\gamma >0$ such that
\begin{align}\label{inequa}\|f(x)-U(x)\|\leq\gamma\varepsilon,\;\text{for\;all\;} x\in X.\end{align}

Based on a result of Gruber \cite{gru}, Gevirtz \cite{gev}
proved that the answer to the Hyers-Ulam problem is positive with
$\gamma=5$. Finally, Omladi\v{c} and \v{S}emrl \cite {Oml} showed that $\gamma=2$ is the sharp constant in (\ref{inequa}).
One can read a long survey of the important topic about the perturbations of isometries on Banach spaces in \cite[page 341-372]{B-L} by Benyamini and Lindenstrauss.

In light of Figiel' theorem, the study of non-surjective
$\varepsilon$-isometries has also brought mathematicians great
attention. Qian \cite {Qia}  proposed the following problem in
1995.

\begin{problem}\label{probl}Does there exist a constant $\gamma>0$ depending only on $X$ and $Y$ with the
following property: For each $\varepsilon$-isometry $f:$
$X\rightarrow Y$ with $f(0)=0$ there is a bounded linear operator
$T:$ $L(f)\rightarrow X$ such that
\begin{align}\label{ineq}
\|Tf(x)-x\|&\leq \gamma\varepsilon,\text{\;\;for\;all }\;x\in X,
\end{align}
where $L(f)$ $=$ $\overline{\rm {span}}f(X)$?
\end{problem}

Then he showed that the answer is affirmative if both $X$ and $Y$
are $L_p$ spaces. \v{S}emrl and V\"{a}is\"{a}l\"{a} \cite {Sem}
further presented a sharp estimate of inequality (\ref{ineq}) with
$\gamma= 2$ if both of them are $L_p$ spaces for $1 < p <\infty$.

The answer to Problem \ref{probl} may be affirmative for some
classical Banach spaces $X$ and $Y$. But Qian
further gave a counterexample.

\begin{example}Given $\varepsilon> 0$, let $Y$ be a separable Banach space admitting a
uncomplemented closed subspace $X$. Assume that g is a bijective
mapping from $X$ onto the closed unit ball $B_Y$ of Y with $g(0)=0$.
We define
$$f:X\rightarrow Y \;by\;f(x)=x+\varepsilon g(x)/2,\text{\;\;for\; all} \;x\in X.$$

Then $f$ is an $\varepsilon$-isometry with $f(0) = 0$ and $ Y=L(f)$.
But there are no such $T$ and $\gamma$ satisfying inequality
(\ref{ineq}).

\end{example}

 Recently, Cheng, Dong and Zhang showed the
following theorem in \cite{Che}.

\begin{theorem}\label{mainlemma} Let $X$ and $Y$ be Banach
spaces, and let $f: X\rightarrow Y$ be an $\varepsilon-$ isometry
with $f(0)=0$ for some $\varepsilon\geq 0$. Then for every $x^*\in
X^*$, there exists $\phi \in Y^*$ with $\|\phi\|=\|x^*\|$ such
that

$$|\langle \phi,f(x)\rangle-\langle
x^*,x\rangle|\leq 4\varepsilon \|x^*\|,\text{\;\;for\;all} \;x\in X.$$

\end{theorem}

In section 2, we introduce some notations and propositions which will be useful
for the proof of our main results, here we refer the interested readers
to \cite[page 19, 102-109]{Phe} and \cite[page 425-516]{Meg} for more details.

In section 3, by using the Rosenthal's $\ell_1$ theorem
and the Cheng-Dong-Zhang theorem (i.e., Theorem \ref{mainlemma}) we
first show that if there is an $\eps$-isometry from a separable
Banach space $X$ into a Banach space $Y$ containing no $\ell_1$,
then there exists a linear isometry from $X$ into $Y^{**}$. As a corollary, we show that the existence of an $\varepsilon$-isometry from a separable Banach space $X$ into $Y$ (the
James space or a reflexive space) implies the existence of a linear
isometry from $X$ into $Y$.

In section 4, we present an equivalent version of Problem \ref{probl} via
continuous linear selections of a set valued mapping, i.e., Problem
\ref{prob} and its weaker solution: Lemma \ref{main1}, by which we study the relationship between differentiability and continuous selections of subdifferential mappings in the setting of $\varepsilon$-isometries (i.e., Proposition \ref{propo}). Finally, we discuss the stability of rotundity and smoothness in Banach spaces under the perturbation by $\varepsilon$-isometries, i.e., [(ii), Proposition \ref{propo}] and Proposition \ref{prop2}.

In this paper, let $X^*$ ( $Y^*$ ) be the dual space of $X$ ( $Y$ ) and $Y^{**}$ be the second dual space of $Y$. We denote $S_X$ ( $S_{X^*}$, $S_{Y^*}$ ), $B_X$ ( $B_{X^*}$ ), $2^{Y}$ ( $2^{X^*}$) by the unit sphere, closed unit ball of $X$ ( $X^*$,  $Y^*$ ), all subsets of $Y$ ( $X^*$ ), respectively.

\section{Preliminaries and notation}

A set valued mapping $F:X \rightarrow 2^{Y}$ is said to be
usco provided it is nonempty compact valued and upper
semicontinuous, i.e., $F(x)$ is nonempty compact for each $x\in X$
and $\{x\in X: F(x)\subset U\}$ is open in $X$ whenever $U$ is open
in $Y$. We say that $F$ is usco at $x\in X$ if $F$ is nonempty
compact valued and upper semicontinuous at $x$, i.e., for every open
set $V$ of $Y$ containing $F(x)$ there exists a open neighborhood $U$
of $X$ such that $F(U)\subset V$. Therefore, $F$ is usco if and only
if $F$ is usco at each $x\in X$.

A mapping $\varphi: X \rightarrow Y$ is called a selection of $F$
if $\varphi(x)\in F(x)$ for each $x\in X$, moreover, we say
$\varphi$ is a continuous (linear) selection of $F$ if $\varphi$ is
a continuous (linear) mapping. We denote the graph of $F$ by
$G(F)\equiv\{(x,y)\in X\times Y:y\in F(x)\}$, we write $F_1\subset
F_2$ if $G(F_1)\subset G(F_2)$. A usco mapping $F$ is said to be minimal
if $E=F$ whenever $E$ is a usco mapping and $E \subset F$.

There are many useful statements about usco mappings and
subdifferential mappings in \cite[page 19, 102-109]{Phe}. In section 3, by using some notions from \cite[page 19, 102-109]{Phe} and combining with the Cheng-Dong-Zhang theorem, we have Proposition \ref{propo} which concerns differentiability and continuous selections of subdifferential mappings.

Recall that a convex function $g$ defined on a nonempty open convex subset $C$ of $X$ is said to be Gateaux
differentiable at $x\in C$ provided that
$\lim_{t\rightarrow 0} \frac{g(x+ty)-g(x)}{t}$
exists for each $y\in X$, which is concerned about a continuous selection of its subdifferential mapping in \cite[Proposition 2.8, page 19]{Phe} as follows:

\begin{proposition} \label{prop1} Suppose that $X$ is a Banach space, $g$ is a continuous
convex function on a nonempty open convex subset $C$ of $X$. Then
$g$ is Gateaux differentiable at each point $x\in C$ if and only if
there is a norm-$w^*$ continuous selection of its subdifferential
mapping $\partial g: C\rightarrow 2^{X^*}$ defined for every $x\in
C$ by
$$\partial g(x)=\{x^*\in X^*: g(y)-g(x)\geq x^*(y-x),\;\text{for\;all\;}y\in C\}.$$
and that $X$ is Gateaux differentiable (smooth) if and only if
$\|\cdot\|$ is Gateaux differentiable at each point of $S_X$ if and
only if $\partial \|\cdot\|$ is single valued at each point of
$S_X$.
\end{proposition}

The following classical results and concepts about rotundity and smoothness of
Banach spaces can be found in \cite[page 425-516]{Meg}.

Recall that

(i)$X$ is said to be rotund if every point in the unit sphere $S_X$ is an extreme point in the closed unit ball $B_X$;

(ii)$X$ is said to be strongly rotund provided the diameter of $C\cap tB_X$ tends to $0$ as $t$ decreases to $d(0,C)$,
whenever $C$ is a nonempty convex subset of $X$.

(iii) $X$ is said to be uniformly Gateaux smooth provided
$\lim_{t\rightarrow 0} \frac{\|x+ty\|-\|x\|}{t}$
exists for each $x\in S_X$ and $y\in X$, and furthermore the
convergence is uniform for $x$ in $S_X$ whenever $y$ is a fixed
point of $S_X$;

(iv) $X$ is said to be Fr\'{e}chet smooth provided the limit in (iii)
exists for each $x\in S_X$ and $y\in X$, and furthermore the
convergence is uniform for $y$ in $S_X$ whenever $x$ is a fixed
point of $S_X$;

(v) $X$ is said to be uniformly Fr\'{e}chet smooth (i.e.,
uniformly smooth) provided the limit in (iii) exists for each $x\in
S_X$ and $y\in X$, and furthermore the convergence is uniform for
$(x,y)$ in $S_X \times S_X$.

 Here we will recall an equivalent definition of $w$ ($w^*$)-uniformly rotund introduced by \v{S}mulian (see \cite[page 464, 466]{Meg}).

\begin{definition} \label{defi} $X$ ($X^*$) is $w$ ($w^*$)-uniformly
rotund whenever $\{x_n\}$ and $\{y_n\}$ are sequences in $S_X$ (
$S_{X^*}$ ) and $\|\frac{1}{2}(x_n+y_n)\|\rightarrow1$, it follows
that $\{x_n-y_n\}$ weakly (resp. weakly star) converges to 0. In
particular, $X$ is said to be uniformly rotund if $\{x_n-y_n\}$ norm-
converges to 0.
\end{definition}

In section 4, we will provide a generalization of Proposition \ref{proposition} in \cite[page 425-516]{Meg}. That is, Proposition \ref{prop2}.

\begin{proposition} \label{proposition} Suppose that $X^*$ is the dual space of $X$. Then

(i) $X$ is rotund (smooth) if $X^*$ is smooth (rotund ); If, in addition, $X$ is reflexive, then the converse also holds;

(ii) $X$ is strongly rotund if and only if $X^*$ is Fr\'{e}chet
smooth;

(iii) $X$ is strongly rotund if and only if $X$ is reflexive, rotund
and has the Radon-Riesz property ( Recall that $X$ has the
Radon-Riesz property if, whenever $\{x_n\}$ is a sequence in $X$ and
$x\in X$ such that $x_n$ weakly converges to $x$ and $\|x_n\|$
converges to $\|x\|$, it follows that $x_n$ strongly converges to
$x$);

(iv) $X$ is weakly uniformly rotund if $X^*$ is uniformly Gateaux
smooth; The converse also holds for every reflexive $X$;

(v) $X$ is uniformly rotund (uniformly smooth) if and only if $X^*$
is uniformly smooth (uniformly rotund).

(vi) $X$ is Fr\'{e}chet smooth if $X^*$ is strongly rotund; If, in
addition, $X$ is reflexive, then the converse also holds;

(vii) $X$ is uniformly Gateaux smooth if and only if $X^*$ is
$w^*$-uniformly rotund;
\end{proposition}

In the following section, we will consider a generalization of Godefroy-Kalton theorem which says that if there exists an isometry from a separable Banach space $X$ into $Y$, then there is a linear isometry from $X$ into $Y$. Indeed, we show that if there is an $\eps$-isometry from a separable
Banach space $X$ into a Banach space $Y$ containing no $\ell_1$,
then there exists a linear isometry from $X$ into $Y^{**}$. That is Theorem \ref{the}, which will be used to prove the main results in section 4.

\section{$\varepsilon$-Isometric embedding into Banach spaces containing no $\ell_1$}

In 2003, Godefroy and Kalton \cite{Gode} studied the relationship
between isometry and linear isometry, and showed the following deep
theorem:
\begin{theorem}(Godefroy-Kalton). \label{GK}
Suppose that $X,Y$ are two Banach spaces. If $X$ is separable and there is an isometry $f: X\rightarrow
Y$, then $Y$ contains an isometric linear copy of $X$;
\end{theorem}

In this section, we will raise an open Problem \ref{problem} and
give another positive example (i.e., Corollary \ref{cor}) for this
problem by using the Rosenthal's $\ell_1$ theorem and Theorem
\ref{mainlemma}.

\begin{problem}\label{problem}Let $f$ be an
$\varepsilon$-isometry from $X$ into $Y$. Does there exist an
isometry from $X$ into $Y$?
\end{problem}

\begin{theorem} \label{the}Let $X$ be a separable Banach space, and let $Y$ be a Banach space such that no closed subspace of $Y$ is isomorphic to $\ell_1$.  If $\varepsilon> 0$, $f$ is an
$\varepsilon$-isometry from $X$ into $Y$, then there is an isometry
from $X$ into $Y^{**}$.
\end{theorem}

\begin{proof} Given $x\in X$, by the Rosenthal's $\ell_1$ theorem
(see \cite{Ros}, \cite[Theorem 10.2.1]{Alb} or \cite[ Theorem 5.37]{Fan}), there exists a weakly Cauchy subsequence
$\left\{\frac{f(n_kx)}{n_k}\right\}_{k=1}^\infty$ of
$\left\{\frac{f(nx)}{n}\right\}_{n=1}^\infty$. Since $Y^{**}$ is $w^*$-semi-complete (Indeed, for every $w^*$-Cauchy sequence $\{y_n^{**}\}_{n=1}^{\infty}$ of $Y^{**}$, let $y^{**}\in Y^{*'}$ (i.e., the Algebraic dual of $Y^*$) be defined for each $y^* \in Y^*$ by $ y^{**}(y^*)=\lim y_n^{**}(y^*)$. So by the uniform boundedness principle $y^{**} \in Y^{**}$), it follows that $\left\{\frac{f(n_kx)}{n_k}\right\}_{k=1}^\infty$ is $w^*$-convergent in $Y^{**}$ (A subset $A$ of a locally convex space is semi-complete if every Cauchy sequence contained in $A$ has a limit in $A$). Let $\{x_m\}_{m=1}^\infty$ be a norm-dense sequence of
$X$.
Then for each $m\in\N$ there is a weakly Cauchy subsequence $$\left\{\frac{f(n^{(m)}_kx_m)}{n^{(m)}_k}\right\}_{k=1}^\infty$$ of $\left\{\frac{f(nx_m)}{n}\right\}_{n=1}^\infty$, and we can inductively choose $\{n_k^{(m)}\}_{k=1}^{\infty}$ such that $\{n_k^{(m+1)}\}_{k=1}^{\infty}$ $\subset$ $\{n_k^{(m)}\}_{k=1}^{\infty}$.
 By a standard diagonal argument,
 we deduce that $$\left\{\frac{f(n^{(k)}_kx_m)}{n^{(k)}_k}\right\}_{k=1}^\infty$$ is also a weakly Cauchy sequence for all $m\in \N$.
 It follows that

$$U(x_m)\equiv w^*-\lim_k \frac{f(n^{(k)}_kx_m)}{n^{(k)}_k} \text{\;\;exists \;for\; all\;} m\in \N.$$
By Theorem \ref{mainlemma}, for each $x^*\in S_{X^*}$, there is a
functional $\phi \in S_{Y^*}$ such that

$$|\langle \phi,f(x)\rangle-\langle
x^*,x\rangle|\leq 4\varepsilon,\;\text{\;for\;all} \;x\in X.$$
Hence

$$\left|\left\langle \phi,\frac{f(n_kx_m)}{n_k}\right\rangle-\langle
x^*,x_m\rangle\right|\leq
\frac{4\varepsilon}{n_k},\;\text{\;for\;all\;}m,\;k\in\N.$$
Letting $k$ tend to
$\infty$, we have

\begin{align}\label{inequ0}\langle \phi, U(x_m)\rangle=\langle
x^*,x_m\rangle,\text{\;\;for\;all} \;m\in \N.\end{align}
 Given $m$, $n \in \N$, by
the Hahn-Banach theorem we can choose a norm-attaining functional
$x^*\in S_{X^*}$ such that

$$\langle x^*,x_m-x_n\rangle=\|x_m-x_n\|.$$
Thus
\begin{align}\label{inequ}
\|x_m-x_n\|&=\langle \phi, U(x_m)\rangle-\langle \phi, U(x_n)\rangle\nonumber\\
&\leq \|U(x_m)-U(x_n)\|.
\end{align}
On the other hand, by the $w^*$-lower semicontinuous argument of a
conjugate norm, we deduce that for every $m$, $n\in\N$
\begin{align}\label{inequ1}
\|U(x_m)-U(x_n)\|&=\left\|w^*-\lim_k\frac{f(n_kx_m)-f(n_kx_n)}{n_k}\right\|\nonumber\\
&\leq \liminf_k \frac{\|n_kx_m-n_kx_n\|+\eps}{n_k}\nonumber\\
&=\|x_m-x_n\|.
\end{align}
Therefore, it follows from (\ref{inequ}) and (\ref
{inequ1}) that $U$ is an isometry from the norm-dense sequence
$(x_m)_{m=1}^\infty$ into $Y^{**}$. Hence $U$ has a unique
extension $\overline{U}: X\rightarrow Y^{**}$ such that
$\overline{U}$ is also an isometry from $X$ into $Y^{**}$.
\end{proof}

\begin{corollary}\label{cor} Let $X$ be a separable Banach space, and let $Y$ be the James space ${\mathcal J}$ or a reflexive space. If $f$ is an
$\varepsilon$-isometry from $X$ into $Y$, then there is a
linear isometry from $X$ into $Y$.
\end{corollary}

\begin{proof} Note that ${\mathcal J}$ is isometric to its
bidual ${\mathcal J^{**}}$ admitting a separable dual but fails to
be reflexive, nowadays known as the James space constructed by James
in \cite{Jam} and \cite{Jam1}(also see \cite[ page 205]{Fan}). By Theorem \ref{GK} and
Theorem \ref{the}, we can easily complete the proof. We would like to emphasize here that an Asplund space (i.e., a space whose dual has the Radon Nikod\'ym property, see \cite[Theorem 5.7, page 82]{Phe}) contains no $\ell_1$, for example, a relexive space or a Banach space with a separable dual.
\end{proof}

\section{Rotundity and smoothness of Banach spaces under the perturbation by
$\varepsilon$-isometries.}

In this section, we consider a set valued mapping version of Problem
\ref{probl} which is equivalent to the following problem and then apply the Cheng-Dong-Zhang theorem to the studies of rotundity and smoothness of Banach spaces under the perturbation by $\varepsilon$-isometries.

\begin{problem}\label{prob}Does there exist a constant $\gamma>0$ depending
 only on $X$ and $Y$ with the following property: For each $\varepsilon$-isometry $f:$ $X\rightarrow Y$ with
$f(0)=0$ there is a $w^*-w^*$ continuous linear selection $Q$ of the
set-valued mapping $\Phi$ from $X^*$ into $2^{L(f)^*}$ defined by
$$\Phi(x^*)\coloneqq \{\phi\in L(f)^*:|\langle
\phi,f(x)\rangle-\langle x^*,x\rangle|\leq\gamma\|x^*\|\varepsilon
,\;\;\text{for\;all}\;x\in X \},$$ where $L(f)$ $=$ $\overline{\text{span}}f(X)$?
\end{problem}

 Now, we present the following set valued mapping versions associated with Problem \ref{probl}
(Problem \ref{prob}), that is, Lemma \ref{main1}, which is very helpful for the proof of our main results.

\begin{lemma}\label{main1}
Suppose that $X$, $Y$ are Banach spaces, $\varepsilon\geq 0$, $r>0$
and $\gamma \geq 4$. Assume that $f$ is a $\varepsilon-$ isometry
from $X$ into $Y$ with $f(0)=0$ and let $\Phi$ be as in Problem \ref {prob}. If  we define a set-valued
mapping $\Phi_r:rB_{X^*}\rightarrow 2^{L(f)^*}$ by
$$\Phi_r(x^*)\coloneqq \{\phi\in rB_{L(f)^*}:|\langle
\phi,f(x)\rangle-\langle x^*,x\rangle|\leq \gamma \|x^*\|\varepsilon
,\;\;\text{for\;all}\;x\in X \},$$ where $L(f)$ $=$ $\overline{\rm{span}}f(X)$. Then

(i) $\Phi_r$ is convex $w^*$-usco at
each point of $rS_{X^*}$.

(ii) There exists a minimal convex norm-$w^*$ usco mapping contained in
$\Phi$.

(iii) If, in addition, $Y$ is separable, then there exists a selection $Q$ of $\Phi$ such that $Q$ is norm-$w^*$ continuous on a norm dense $G_\delta$ set of $X^*$.
\end{lemma}

\begin{proof} (i) By the definition of $\Phi_r$ and the Cheng-Dong-Zhang theorem it is clear that $\Phi_r$ is nonempty, convex and $w^*$-compact valued. Now, we will show that
it is $w^*$-$w^*$ upper semicontinuous at each $x^*\in rS_{X^*}$. Let
$(x^*_\alpha)_{\alpha \in \Gamma}\subset rB_{X^*}$ be a net $w^*$
convergent to $x^*\in rS_{X^*}$ and $y^*_\alpha\in
\Phi_r(x^*_\alpha)$ for all $\alpha \in \Gamma$. By Alaouglu
theorem, there exists a subnet $(y^*_\beta)\subset (y^*_\alpha)$
$w^*$- convergent to some $y^*\in rB_{L(f)^*}$ such that for every $x\in X$,

$$|\langle y^*_\beta,f(x)\rangle-\langle x^*_\beta,x\rangle|\leq \gamma r \varepsilon.$$
 Hence for every $x\in X$, by taking limit with respect to $\beta$ we have

$$|\langle y^*,f(x)\rangle-\langle x^*,x\rangle|\leq \gamma r \varepsilon,$$
 which yields $y^*\in \Phi_r(x^*)$. Therefore, $\Phi_r $ is
$w^*-w^*$ upper semicontinuous at each point $x^*$ of $rS_{X^*}$ (If
not, by the definition of a usco mapping for some $w^*$-open set
$U\supset \Phi_r(x^*)$, we can find a net $(x^*_\alpha) $
$w^*$-convergent to $x^* \in rS_{X^*}$ such that for every $\alpha \in \Gamma$, there
exist $y^*_\alpha\in \Phi_r(x^*_\alpha)$ and $y^*_\alpha\notin U$. Since
$y^*_\alpha \notin U$ for all $\alpha \in \Gamma$, it is impossible that
any subnet of it $w^*$-converges to some $y^*\in\Phi_r(x^*)$).

 (ii) Let $F:X^*\rightarrow
2^{L(f)^*}$ be defined for all $x^*\in
X^*$ by
$$F(x^*)\coloneqq \{\phi\in \Phi(x^*):\|\phi\|=\|x^*\|\}.$$
Hence, by the Cheng-Dong-Zhang theorem (i.e., Theorem \ref{mainlemma}) for each $x^*\in X^*$,
$F(x^*)$ is a nonempty, convex and $w^*$-compact subset of $L(f)^*$
and $F\subset \Phi$. Thus, it suffices to show that $F$ is norm
$-w^*$ upper semicontinuous and hence by Zorn Lemma (see \cite[Proposition 7.3, page 103]{Phe}) there exists a
minimal convex norm$-w^*$ usco mapping contained in $\Phi$.

Let $\{x^*_n\}$ be a sequence convergent to
$x^*\in X^*$ in its norm- topology. By the
definition of $F$, for each $y^*_n\in F(x^*_n)$ we
have $\|y^*_n\|=\|x^*_n\|$ for all $n$. By the
$w^*$-compactness argument, there exists a subnet
$(y^*_\beta)\subset (y^*_n)$ $w^*$- convergent to some $y^*\in
L(f)^*$ and it follows that $y^*\in F(x^*)$. Therefore, by using (i) again $F$ is norm-$w^*$ upper semicontinuous at each point $x^*\in X^*$.

(iii) By (ii) there is a minimal convex norm-$w^*$ usco mapping $F'\subset F\subset\Phi$, and note that $X^*$ is a Baire space  and there exists a norm-dense countable set $\{x_n\}_{n=1}^{\infty} \subset S_{L(f)}$ such that the relative $w^*$-topology on every bounded subset $A$ of $L(f)^*$ coincides with a metric defined for all $x^*$, $y^*\in X^*$ by $$d(x^*,y^*)=\sum_{n=1}^{\infty}2^{-n}|\langle x^*-y^*,x_n\rangle|,$$
which follows easily from \cite[Lemma 7.14, page 106-107]{Phe}.
\end{proof}

Combining Lemma \ref{main1}, Theorem \ref{the} with the Cheng-Dong-Zhang theorem, we have the following two propositions about
rotundity and smoothness of Banach spaces under the perturbation by
$\varepsilon$-isometries. Then our results cover some classical conclusion if we come to the special case that $f$ is the identity and $X=Y$.

\begin{proposition} \label{propo}Suppose that $X$, $Y$ are Banach spaces, $\varepsilon\geq 0$, and let $f$ be
an $\varepsilon$-isometry from $X$ into $Y$ with $f(0)=0$. Let
$\Phi_1$ be  as in Lemma \ref{main1}. Then

(i) $X$ is smooth if there is a norm-$w^*$ continuous selection of
$\Phi_1\circ\partial\|\cdot\| :X\rightarrow 2^{L(f)^*}$.

(ii) In particular, if $Y^*$ is rotund, then $X^*$ is also rotund. Hence $X$ is smooth.
\end{proposition}

\begin{proof} (i) Assume that $\phi:X\rightarrow L(f)^*$ is a norm-$w^*$ continuous
selection of $\Phi_1\circ\partial\|\cdot\|$, that is, for every
$x\in X$, there is $x^*\in \partial\|x\|$ such that $\phi(x)\in
\Phi_1(x^*)$.  In fact, for two functionals $x_1^*$, $x_2^*\in S_X^*$ satisfying $\varphi(x_1^*)=\varphi(x_2^*)$, we have $x_1^*=x_2^*$ by triangle inequality. That is, for every $x\in X$,
$$|\langle x^*_1,x\rangle-\langle x^*_2,x\rangle|\leq|\langle \varphi(x^*_1),f(x)\rangle-\langle x^*_1,x\rangle|+|\langle \varphi(x^*_2),f(x)\rangle-\langle x^*_2,x\rangle|\leq 2\gamma \varepsilon,$$ which implies $
x^*_1-x^*_2=0$.
Since every selection $\varphi$ of  $\Phi_1$ is injective, if
$\varphi$ is a selection of $\Phi_1$, then
$\varphi^{-1}\circ\phi:X\rightarrow S_{X^*}$ is a selection of
$\partial\|\cdot\|$. Hence by Proposition \ref{prop1} it suffices to
show that $\varphi^{-1}\circ\phi$ is norm-$w^*$ continuous. Let
$\{x_n\}\subset X$ be a sequence  norm-converging to $x_0\in X$. By assumption, for
each $n \in \mathbb{N}$ there is $x^*_n\in \partial\|x_n\|$ such
that $\phi(x_n)=\varphi(x^*_n)$ and $(\varphi(x^*_n))$ $w^*$-converges to $\varphi (x^*_0)$. It remains to show that $(\varphi^{-1}\circ\phi(x_n))$ is $w^*$-convergent to $\varphi^{-1}\circ\phi(x_0)$.

On one hand, it follows from the definition of $\Phi_1$ and the Cheng-Dong-Zhang theorem that for every $x\in X$,
\begin{align}\label{inequal}
|\langle \varphi(x^*_n),f(x)\rangle-\langle x^*_n,x\rangle|&\leq \gamma \varepsilon,
\end{align}
and
$$|\langle \varphi(x^*_0),f(x)\rangle-\langle x^*_0,x\rangle|\leq \gamma \varepsilon.$$
On the other hand, for every subnet $\{x^*_\alpha\}$ of $\{x^*_n\}$, by Alaouglu theorem there exists a $w^*$-convergent subnet $\{x^*_\beta\}$ contained in $\{x^*_\alpha\}$. Since every selection of $\Phi_1$ is injective, by substituting $\beta$ for $n$ and taking limit with respect to $\beta$ in \ref{inequal} we deduce that $w^*-\lim_\beta x^*_\beta=x^*_0$.
Therefore, $\{x^*_n\}$ is $w^*$-convergent to $x^*_0$ and hence $X$ is
smooth.

(ii)By Lemma \ref{main1} $\Phi_1$ is convex $w^*$-usco at
each point of $S_{X^*}$. Note that the subdifferential mapping $\partial\|\cdot\|$ is convex norm-$w^*$ usco. Thus the compound $\Phi_1\circ\partial\|\cdot\|$ is convex norm-$w^*$ usco. By (i) it suffices to show that $\Phi_1\circ\partial\|\cdot\|$ is single valued. If $Y^*$ is rotund, then by the Hahn-Banach
theorem every point of $S_{L(f)^*}$ is an extreme point of
$B_{L(f)^*}$. Therefore, we can deduce that $\Phi_1\circ\partial\|\cdot\|$ is
single valued at each point of $S_{X^*}$. ( In fact, if for some $x\in X$ and $x_1^*,x_2^*\in \partial\|x\|$, there exist double functionals $\phi(x_1^*)$, $\phi(x_2^*)\in \Phi_1\circ\partial\|x\|$, then every convex combination $\lambda \phi(x_1^*)+(1-\lambda)\phi(x_2^*)$ $\in \Phi_1(\lambda x_1^*+(1-\lambda)x_2^*)$ for each $0<\lambda<1$, and hence $\|\lambda \phi(x_1^*)+(1-\lambda)\phi(x_2^*)\|=1$ which is a contradiction.) Hence by the conclusion of (i) $X$ is smooth (By using the similar reasoning $X^*$ is even rotund).
\end{proof}

\begin {remark} Note that the converse of (i) in Proposition \ref{propo} also holds whenever $\Phi_1$ admits a $w^*-w^*$ continuous selection (In particular, if $Y^*$ is rotund). However, we don't know whether it also holds in general case.
\end{remark}

\begin{proposition}\label{prop2}Suppose that $X$, $Y$ are Banach spaces, $\varepsilon\geq 0$, $f$ is
an $\varepsilon$-isometry from $X$ into $Y$ with $f(0)=0$. Then

(i) $X$ is rotund if $Y^*$ is smooth;

(ii) $X$ is weakly uniformly rotund if $Y^*$ is uniformly Gateaux
smooth;

(iii) $X$ is strongly rotund if $Y^*$ is Fr\'{e}chet smooth;

(iv) $X$ is Fr\'{e}chet smooth if $Y^*$ is strongly rotund;

(v) $X$ is uniformly rotund if $Y^*$ is uniformly smooth;

(vi) $X$ is uniformly smooth if $Y^*$ is uniformly rotund.

\end{proposition}

\begin{proof}(i) Let $\varphi$ be a selection of $\Phi_1$ such that for all $x\in X$ and $x^*\in
S_{X^*}$,
\begin{align}\label{inequality}
|\langle \varphi(x^*),f(x)\rangle-\langle x^*,x\rangle|&\leq
4\varepsilon.
\end{align}
And let $M$ be defined by
$$M=\overline{span}\{\varphi(x^*): x^*\in S_{X^*}\}.$$
Since for every $x\in X$, $\{\frac{f(nx)}{n}\}$ is a norm-bounded
sequence of $Y$, it follows from (\ref{inequality}) and uniform
boundedness principle that the following limit exists for every
$m\in M$,
$$\langle U(x),m\rangle =\lim_n \left\langle \frac{f(nx)}{n},m\right\rangle ,$$ where $U:X\rightarrow M^*$ is well defined for
all $x\in X$ by
$$U(x)\equiv w^*-\lim_n \frac{f(nx)}{n}.$$
By an analogous proof of Theorem \ref{the}, $U$ is an isometry from
$X$ into $M^*$ such that for each $x^*\in S_{X^*}$ and $x\in X$,
\begin{align}\label{equality}
\langle \varphi(x^*), U(x)\rangle&=\langle x^*,x\rangle.
\end{align}
Therefore, if $Y^*$ is smooth (In fact, $U$ is linear), then
$M$ as a subspace of $Y^*$ is also smooth. Hence by the above
equality (\ref{equality}) $X$ is rotund ( If not, then there are double points $x_1,x_2\in S_X$ and $x^*\in S_{X^*}$ such that $\langle x^*,x_1\rangle$ $=\langle x^*,x_2\rangle=1$. Hence by the equality \ref{equality} we have $\langle \varphi(x^*), U(x_1)\rangle =\langle \varphi(x^*), U(x_2)\rangle=1$ which is a contradiction with the smoothness of $M$).

(ii)Assume that  $Y^*$ is uniformly Gateaux smooth. Let
$\{x_n\}_{n=1}^\infty$, $\{y_n\}_{n=1}^\infty$ be two sequences of
$S_X$ such that
$$\lim_n\left\|\frac{x_n+y_n}{2}\right\|=1.$$ By definition it suffices to show
for every $x^*\in S_{X^*}$ that
\begin{align}\label{e1}
\lim_n \langle x^*,x_n-y_n\rangle=0.
\end{align}
It first follows from (i), the assumption and (vii) in Proposition
\ref{proposition} that $Y^{**}$ is $w^*$-uniformly rotund and
\begin{align}\label{e}
\lim_n\left\|\frac{U(x_n)+U(y_n)}{2}\right\|&=\lim_n\|\frac{x_n+y_n}{2}\|=1,
\end{align}
then we deduce that for every $n\in\mathbb{ N }$  $U(x_n)$, $U(y_n)$ and
$\frac{U(x_n)+U(y_n)}{2}$ have a unique norm-preserving extension
from $M^*$ to $Y^{**}$ denoted by $\overline{U(x_n)}$,
$\overline{U(y_n)}$ and $\overline{\frac{U(x_n)+U(y_n)}{2}}$,
respectively such that
\begin{align}\label{e2}
\overline{\frac{U(x_n)+U(y_n)}{2}}&=\frac{\overline{U(x_n)}+\overline{U(y_n)}}{2}.
\end{align}
Finally, it follows from definition \ref{defi}, equality (\ref{equality}),
(\ref{e}) and (\ref{e2}) that for every $x^*\in S_{X^*}$,

$$
\lim_n \langle x^*,x_n-y_n\rangle=\lim_n\langle \varphi(x^*),
U(x_n)-U(y_n)\rangle$$
$$\;\;\;\;\;\;\;\;\;\;\;\;\;\;\;\;\;\;\;\;\;\;\;\;\;\;\;\;\;\;=\lim_n\langle \varphi(x^*), \overline{U(x_n)}-\overline{U(y_n)}\rangle=0.$$
Hence (\ref{e1}) holds, and by Definition \ref{defi} $X$ is weakly
uniformly rotund.

(iii-vi) It follows from the assumptions of (iii-vi) that $Y$
is reflexive. Thus we can easily deduce from (i) and Proposition
\ref{proposition} that $M^*$ of (i) is  strongly rotund, Fr\'{e}chet smooth, uniformly
rotund and uniformly smooth, respectively. Hence $X$ is strongly rotund, Fr\'{e}chet smooth, uniformly
rotund and uniformly smooth, respectively.
\end{proof}

\begin{fact} \label{fact} A Banach space $X$ is uniformly smooth  (resp. admitting Radon Riesz property, reflexive, rotund, smooth, Fr\'{e}chet smooth, strongly rotund, uniformly rotund) if and only if so is every separable subspace of $X$.
\end{fact}

\begin{remark}
 Since $Y$ in Proposition \ref{prop2} (iii-v) is reflexive, we can also apply Theorem \ref{the}, indeed Corollary \ref{cor} and the above classical fact to prove Proposition \ref{prop2} (iii-vi) easily and immediately. In fact, we can prove the above classical Fact \ref{fact} by definition. We take the uniform smoothness for an example, it suffices to show that $X$ is uniformly smooth if every separable subspace of it is uniformly smooth (Obviously, this assumption implies that $X$ is smooth). If it is not uniformly smooth, then there exists a sequence $\{(x_n,y_n)\}_{n=1}^\infty\subset S_X \times S_X$ such that the limit of the definition (v) of section 2 exists but not uniformly for $\{(x_n,y_n)\}_{n=1}^\infty\subset S_X \times S_X$. Then $\overline{\rm span}{\{x_n, y_n\}_{n=1}^{\infty}}$ is not uniformly smooth, which is a contradiction.
\end{remark}

We leave open the following questions about
$\varepsilon$-isometric embeddings.

\begin{problem}Suppose that $X$, $Y$ are Banach spaces, $\varepsilon> 0$,  and $f$ is
an $\varepsilon$-isometry from $X$ into $Y$ with $f(0)=0$.

(i) Does there exist an isometry from $X$ into $Y$?

(ii) Can we characterize the space $X$ ($Y$) satisfying that for
every $Y$ ($X$), if such $f$ exists, then there is an isometry from
$X$ into $Y$?

(iii) If, in addition, $Y$ has some property ($P$) (for example,
smoothness, rotundity), so does $X$?

\end{problem}

{\bf Acknowledgements.}

The authors thank the referee for many helpful and kind suggestions. This work was partially done while the first author was visiting Texas A$\&$M University and in Analysis and Probability Workshop at Texas A$\&$M University which was funded by NSF Grant. The first author would like to thank Professor Thomas Schlumprecht for the invitation. Both authors also expresses their appreciation to Professor Lixin Cheng for very helpful comments.

The first author was supported by the China Scholarship Council, the second author was supported by the Natural Science Foundation of China, grant 11201353.

\bibliographystyle{amsplain}

\end{document}